\documentclass[12pt]{article}
\usepackage{amssymb}
\usepackage{array}
\usepackage{latexsym}
\usepackage{enumerate}
\usepackage{amsmath}
\usepackage{amsfonts}
\usepackage{amsthm}
\usepackage[english]{babel}

\title{\bf{A characterization of the Artin-Mumford curve}}
\author{N.~Arakelian\footnote{This paper was written when the first author was visiting the Universt\`a Degli Studi Della Basilicata. This visit was supported by FAPESP-Brazil (grant 2014/03497-6). } \   and G.~Korchm\'aros}

\newtheorem{theorem}{Theorem}[section]
\newtheorem{proposition}[theorem]{Proposition}
\newtheorem{lemma}[theorem]{Lemma}
\newtheorem{corollary}[theorem]{Corollary}

\theoremstyle{definition}
\newtheorem*{definition*}{Definition}

\newtheorem*{proposition*}{Proposition}
\newtheorem*{corollary*}{Corollary}
\newtheorem*{lemma*}{Lemma}

\def\cC{\mathcal C}

\def\cF{\mathcal F}

\def\cH{\mathcal H}

\def\cM{\mathcal M}

\def\cX{\mathcal X}
\def\cY{\mathcal Y}

\def\l{\ell}
\def\Aut{\mbox{\rm Aut}}
\def\Gal{\mbox{\rm Gal}}

\def\deg{\mbox{\rm deg}}

\def\Aut{\mbox{\rm Aut}}

\newcommand{\PGL}{\mbox{\rm PGL}}

\newcommand{\GL}{\mbox{\rm GL}}
\newcommand{\aut}{\mbox{\rm Aut}}

\newcommand{\g}{\gamma}

\begin{document}
\maketitle

\begin{abstract} Let $\cM$ be the Artin-Mumford curve over the finite prime field $\mathbb{F}_p$ with $p>2$. By a result of Valentini and Madan, $\aut_{\mathbb{F}_p}(\cM)\cong H$ with $H=(C_p\times C_p)\rtimes D_{p-1}$. We prove that if $\cX$ is an algebraic curve of genus $\textsf{g}=(p-1)^2$ such that $\aut_{\mathbb{F}_p}(\cX)$ contains a subgroup isomorphic to $H$ then $\cX$ is birationally equivalent over $\mathbb{F}_p$ to the Artin-Mumford curve $\cM$.
\end{abstract}
%


\section{Introduction}\label{s1}
The Artin-Mumford curve $\cM_c$ over a field $\mathbb{F}$ of characteristic $p>2$ is the (projective, nonsingular, geometrically irreducible) algebraic curve birationally equivalent over $\mathbb{F}$ to the plane curve with affine equation
\begin{equation}
\label{31octa}
(x^p-x)(y^p-y)=c, \ \ \ \ c \in \mathbb{F}^{*}.
\end{equation}
It is known that $\cM_c$ has genus $\textsf{g}=(p-1)^2$ and that its $\mathbb{F}$-automorphism group
$\aut_{\mathbb{F}}(\cM_c)$ contains a subgroup
\begin{equation}\label{gro}
H=(C_p\times C_p)\rtimes D_{p-1},
\end{equation}
where $C_p$ is a cyclic group of order $p$, $D_{p-1}$ is a dihedral group of order $2(p-1)$, and $(C_p\times C_p)\rtimes D_{p-1}$ is the semidirect product of $C_p\times C_p$ by $D_{p-1}$.

Artin-Mumford curves, especially over non-Archimedean valued fields of positive characteristic, have been investigated in several papers; see \cite{Co-Ka2003-1,Co-Ka-Ko2001,Co-Ka2004}. By a result of Cornelissen, Kato and Kontogeorgis \cite{Co-Ka-Ko2001}, $\aut_{\mathbb{F}}(\cM_c)=H$ for any non-Archimedean valued field $(\mathbb{F}, |\cdot|)$ of positive characteristic, provided that $|c|<1$. Valentini and Madan \cite{valentinimadan} proved the same result  $\aut_{\mathbb{F}}(\cM_c)=H$ if $\mathbb{F}$ is an algebraically closed field of characteristic $p>0$.

In this paper we study the $\mathbb{F}$-automorphism group of the Artin-Mumford curve $\cM_c$ in the case where $\mathbb{F}$ is the finite prime field $\mathbb{F}_p$ with $p>2$. Since $\cM_c$ is $\mathbb{F}_p$-birationally equivalent to $\cM_1$ for all $c \in \mathbb{F}_p^{*}$, we set $\cM=\cM_1$. Our main result stated in the following theorem shows that $\Aut_{\mathbb{F}_p}(\cM)$ characterizes $\cM$.

\begin{theorem}\label{main}
Let $\cX$ be a (projective, nonsingular, geometrically irreducible algebraic) curve of genus $\textsf{g}=(p-1)^2$ defined over $\mathbb{F}_p$ with $p>2.$ If $\Aut_{\mathbb{F}_p}(\cX)$ contains a subgroup $H$ isomorphic to (\ref{gro}) then $\cX$ is birationally equivalent over $\mathbb{F}_p$ to the Artin-Mumford curve $\cM$.
\end{theorem}


\section{Preliminary Results}\label{bg}
Notation and terminology are standard. Our principal reference is \cite{HKT}. The concepts and results we use can also be found in \cite{stichtenothbook}. In particular, we use the term of a curve to denote a projective, nonsingular, geometrically irreducible algebraic curve.

In this paper, $p$ is an odd prime number, $\mathbb{F}_p$ is the finite field of order $p$, $\cX$ is a curve defined over $\mathbb{F}_p$ and viewed as a curve over the algebraic closure $\overline{\mathbb{F}}_p$ of $\mathbb{F}_p$. Also, $\mathbb{F}_p(\cX)$ is the function field of $\cX$ over $\mathbb{F}_p$, and $\Aut_{\mathbb{F}_p}(\cX)$ is its automorphism group over $\mathbb{F}_p$. Furthermore, referring to $\mathbb{K}$ we replace $\mathbb{F}_p$ by $\mathbb{K}$ in our notation, where $\mathbb{K}$ denotes an algebraically closed field of odd characteristic $p$.

For a positive integer $d$, $C_d$ stands for a cyclic group of order $d$, and $D_d$ for a dihedral group of order $2d$.

For a subgroup $S$ of $\Aut_{\mathbb{K}}(\cX)$, the fixed field of $S$ is the subfield $\mathbb{K}(\cX)^{S}$ of $\mathbb{K}(\cX)$ which consists of all elements of $\mathbb{K}(\cX)$ fixed by every automorphism in $S$. The quotient curve $\cX/S$ of $\cX$ by $S$ is defined, up to birational equivalence, to be the curve with function field $\mathbb{K}(\cX)^{S}$. It is well known that the field extension $\mathbb{K}(\cX)|\mathbb{K}(\cX)^{S}$ is a Galois extension of degree $|S|$ with Galois group $S=\Gal(\mathbb{K}(\cX)|\mathbb{K}(\cX)^{S})$.

For $P \in \cX$, the orbit of $P$ is the set
\begin{equation}\label{orbit}
S(P)=\{\sigma(P) \ | \ \sigma \in S\}.
\end{equation}
and it is long if $|S(P)|=|S|$, otherwise short. If $\bar{P}_1, \ldots,\bar{P}_k, \in \cX/S$ are the ramified points of the cover $\cX \rightarrow \cX/S$, then the short orbits of $S$ are exactly the pointsets $\Omega_i$ of $\cX$ which lie over $\bar{P}_i$ where $i=1,\ldots,k$. All but finitely many orbits of $S$ are long. It may happen that $S$ does not have any short orbit. This case only occurs when no nontrivial element in $S$ has a fixed point in $\cX$, equivalently the extension $\mathbb{K}(\cX)|\mathbb{K}(\cX)^{S}$ is unramified.
The stabilizer of $P \in \cX$ in $S$ is the subgroup defined by
\begin{equation}\label{stab}
S_P=\{\sigma \in S \ | \ \sigma(P)=P\}.
\end{equation}
For a non-negative integer $i$, the $i$-th ramification group $S_P^{(i)}$ of $\cX$ at $P$ is
\begin{equation}\label{ithram}
S_P^{(i)}=\{\sigma \in S_P \ | \ v_P(\sigma(t)-t) \geq i+1\},
\end{equation}
where $t$ is a local parameter at $P$ and $v_P$ denotes the discrete valuation at $P$. Here $S_P=S_P^{(0)}$. Furthermore, $S_P^{(1)}$ is a normal $p$-subgroup of $S_P^{(0)}$, and the factor group $S_P^{(0)}/S_P^{(1)}$ is cyclic of order prime to $p$; see e.g. \cite[Theorem 11.74]{HKT}. In particular, if $S_P$ is a $p$-group, then $S_P=S_P^{(0)}=S_P^{(1)}$.

Let $\textsf{g}$ and $\bar{\textsf{g}}$ be the genus of $\cX$ and $\bar{\cX}=\cX/S$, respectively. The Riemann-Hurwitz genus formula is \begin{equation}\label{rhg}
2\textsf{g}-2=|S|(2\bar{\textsf{g}}-2)+\sum_{P \in \cX}\sum_{i \geq 0}\big(|S_P^{(i)}|-1\big);
\end{equation}
see \cite[Theorem 11.72]{HKT}.
 If $\l_1,\ldots,\l_k$  are the sizes of the short orbits of $S$, then (\ref{rhg}) yields
\begin{equation}\label{rhso}
2\textsf{g}-2 \geq |S|(2\bar{\textsf{g}}-2)+\sum_{i=1}^{k} \big(|S|-\l_i\big),
\end{equation}
and equality holds if $\gcd(|S_P|,p)=1$ for all $P \in \cX$; see \cite[Theorem 11.57 and Remark 11.61]{HKT}.

Let $\g$ be the $p$-rank of $\cX$. Then $0\le \g \leq \textsf{g}$, and when equality holds $\cX$ is \emph{ordinary} or \emph{general}; see \cite[Theorem 6.96]{HKT}.
For a $p$-subgroup of $\Aut_{\mathbb{K}}(\cX)$, the Deuring-Shafarevich formula is
\begin{equation}\label{dsg}
\g-1=|S|(\bar{\g}-1)+\sum_{i=1}^{k} \big(|S|-\l_i\big);
\end{equation}
where $\bar{\g}$ is the $p$-rank of $\bar{\cX}=\cX/S$; see \cite[Theorem 11.62]{HKT}.

The result below follows from \cite{geervlugt1991} and provides a characterization of certain Artin-Schreier extensions of the projective line over $\mathbb{F}_q$, where $q$ is a power of $p$.

\begin{theorem}[van der Geer - van der Vlugt]\label{tvgvv}
 The curves over $\mathbb{F}_q$ with affine equation
\begin{equation}\label{vgvv}
y^p-y=ax+\frac{1}{x},
\end{equation}
where $a \in \mathbb{F}_q^{*}$, are exactly the curves $\cC$ of genus $p-1$ defined over $\mathbb{F}_q$ with the properties:
\begin{itemize}
\item $\cC$ is hyperelliptic.
\item $\Aut_{\mathbb{F}_q}(\cC)$ contains a subgroup $\langle \iota \rangle \times D_{p}$, where $\iota$ is a hyperelliptic involution and $D_{p}$ is a dihedral group of order $2p$.
\item The fixed points of $D_p$ are  $\mathbb{F}_q$-rational.
\end{itemize}
\end{theorem}
We remark that the hypothesis on $\cC$ to be hyperelliptic can be dropped; see Proposition \ref{pr3oct}.


\section{Background on the automorphism group of the Artin-Mumford curve}
\label{bgAM}

The \emph{Artin-Mumford curve} is the curve defined by the affine equation
\begin{equation}\label{amc}
\cM:(x^p-x)(y^p-y)=1.
\end{equation}
By \cite[Theorem 7]{valentinimadan} the automorphism group ${\Aut_{\overline{\mathbb{F}}_p}(\cM)}$ is $(C_p \times C_p)\rtimes D_{p-1}$, and is generated by
$$
\tau_{a,b}:(x,y)\mapsto(x+a,y+b) \ \ \text{ with } a,b \in \mathbb{F}_p,
$$
$$
\mu:(x,y) \mapsto (y,x),
$$
and
$$
\theta_d:(x,y) \mapsto (dx,d^{-1}y) \ \ \text{ with } d \in \mathbb{F}_p^{*}.
$$
\begin{proposition}\label{properties}
Let $\cM:(x^p-x)(y^p-y)=1$ be the Artin-Mumford curve of genus $\textsf{g}$ and $p$-rank $\g$. Then
\begin{itemize}
\item [(i)] $\textsf{g}=\g=(p-1)^2$.
\item [(ii)] $\# \cM(\mathbb{F}_p)=2p$.
\item [(iii)] The quotient curves $\cM/ \langle \tau_{1,0} \rangle$ and $\cM/ \langle \tau_{0,1} \rangle$ are rational curves.
\end{itemize}
\end{proposition}
\begin{proof} (i) The curve $\cM$ has exactly two singularities, $O_1=(1:0:0)$ and $O_2=(0:1:0)$, both ordinary $p$-fold points.
Thus,
$$
\textsf{g}=\frac{(2p-1)(2p-2)}{2}-\frac{2p(p-1)}{2}=(p-1)^2.
$$
Moreover, by \cite[Proposition 3.2]{subrao1975}, is ordinary and hence $\textsf{g}=\g=(p-1)^2$.

(ii) Obviously, $\cM$ has no affine $\mathbb{F}_p$-rational points. The tangent lines to $\cM$ at $O_1$ are
$y=a$ with $a\in \mathbb{F}_p$, while those at $O_2$ have equation $x=b$ with $b\in \mathbb{F}_p$. Therefore (ii) holds.

(iii) Set $\eta=x^p-x$. Then $\mathbb{F}_p(\eta,y) \subseteq \mathbb{F}_p(\cM)^{\langle \tau_{1,0} \rangle}$, and hence $\mathbb{F}_p(\cM)=\mathbb{F}_p(\eta,y,x)$, which yields $[\mathbb{F}_p(\cM):\mathbb{F}_p(\eta,y)]=p$. Since $|\langle \tau_{1,0} \rangle|=p$, we have $\mathbb{F}_p(\eta,y) = \mathbb{F}_p(\cM)^{\langle \tau_{1,0} \rangle}$. Hence
$$
\cM/ \langle \tau_{1,0} \rangle:y^p-y=\frac{1}{\eta}.
$$
This shows that $\cM/ \langle \tau_{1,0} \rangle$ is rational. The same holds for $\cM/ \langle \tau_{0,1} \rangle$.
\end{proof}


\section{Curves with automorphism group containing $(C_p\times C_p)\rtimes D_{p-1}$}
\label{gener}

In this section we state and prove the additional results that will be used in the proof of Theorem \ref{main}. We begin with a technical lemma.

\begin{lemma}
\label{le1oct} Let $D_{p-1}=\langle U,V \rangle$ be a subgroup of $\GL(2,p)$ isomorphic to a dihedral group of order $2(p-1)$ where $U^{p-1}=V^2=I$, and $VUV=U^{-1}$. Then, up to conjugacy in $\GL(2,p)$,
$$U=\left(
  \begin{array}{cc}
    \lambda & 0 \\
    0 & \lambda^{-1} \\
  \end{array}
\right), \qquad
V=\left(
  \begin{array}{cc}
    0 & 1 \\
    1 & 0 \\
  \end{array}
\right),
$$
where $\lambda$ is a primitive element of $\mathbb{F}_p$.
\end{lemma}
\begin{proof} Let $\PGL(2,p)=\GL(2,p)/Z(\GL(2,p))$ and regard it as the automorphism group of the rational function field $\mathbb{F}_p(x)$. The center $Z(D_{p-1})$ consists of $W=U^{(p-1)/2}$ and $I$. Therefore, either $D_{p-1}\cap Z(\GL(2,p))=I$ or $D_{p-1}\cap Z(\GL(2,p))=\{I,W\}$.

In the former case, $D_{p-1}$ is isomorphic to a subgroup $M$ of  $\PGL(2,p)$.
{}From the classification of subgroups of $\PGL(2,q)$ up to conjugacy, $M=\langle u,v \rangle$ with
$u: x\mapsto\lambda x$ and $v:x\mapsto x^{-1}$ where $\lambda$ is a primitive element of $\mathbb{F}_p$; see \cite[Theorem A.8]{HKT} for instance. Therefore, there exist $\alpha, \beta \in \mathbb{F}_p^*$ such that, up to conjugacy in $\GL(2,p)$,
$$U=\alpha \left(
  \begin{array}{cc}
    \lambda & 0 \\
    0 & 1 \\
  \end{array}
\right), \qquad
V=\beta \left(
  \begin{array}{cc}
    0 & 1 \\
    1 & 0 \\
  \end{array}
\right).
$$
Here $\beta=\pm 1$ as $V^2=I$. Furthermore,
$VUVU=I$ holds if and only if
$$\alpha^2\left(
  \begin{array}{cc}
    1 & 0 \\
    0 & \lambda \\
  \end{array}
\right)
\left(
  \begin{array}{cc}
    \lambda & 0 \\
    0 & 1 \\
  \end{array}
\right) = I
$$
whence $\alpha^2\lambda=1$ which is impossible as $\lambda$ is a primitive element of $\mathbb{F}_p$.

Therefore, $D_{p-1}\cap Z(\GL(2,p))=\{I,W\}$ and hence $D_{p-1}/\langle W \rangle$ is isomorphic to a dihedral subgroup of $\PGL(2,p)$ of order $p-1$. This time, from the classification of subgroups of $\PGL(2,q)$ up to conjugacy, $M=\langle u,v \rangle$ with
$u: x\mapsto \lambda^2 x$ and $v:x\mapsto x^{-1}$ where $\lambda$ is a primitive element of $\mathbb{F}_p$. Arguing as before, this implies the existence of $\alpha \in \mathbb{F}_p^*$ such that, up to conjugacy in $\GL(p)$,
$$U=\alpha \left(
  \begin{array}{cc}
    \lambda^2 & 0 \\
    0 & 1 \\
  \end{array}
\right), \qquad
V=\pm \left(
  \begin{array}{cc}
    0 & 1 \\
    1 & 0 \\
  \end{array}
\right).
$$
 Furthermore,
$VUVU=I$ holds if and only if
$$\alpha^2\left(
  \begin{array}{cc}
    1 & 0 \\
    0 & \lambda^2 \\
  \end{array}
\right)
\left(
  \begin{array}{cc}
    \lambda^2 & 0 \\
    0 & 1 \\
  \end{array}
\right) = I,
$$
that is, $\alpha^2\lambda^2=1$. Therefore,
$$U=\left(
  \begin{array}{cc}
    \alpha^{-1} & 0 \\
    0 & \alpha \\
  \end{array}
\right), \qquad
V=\pm \left(
  \begin{array}{cc}
    0 & 1 \\
    1 & 0 \\
  \end{array}
\right).
$$
In particular, $\alpha$ is a primitive element of $\mathbb{F}_p$. If $$V=- \left(
  \begin{array}{cc}
    0 & 1 \\
    1 & 0 \\
  \end{array}
\right),
$$
then replace $V$ by $U^{(p-1)/2}V$. Then $V$ becomes  $$V=\left(
  \begin{array}{cc}
    0 & 1 \\
    1 & 0 \\
  \end{array}
\right),
$$
and still $D_{p-1}=\langle U,V \rangle$. This completes the proof.
\end{proof}

\begin{corollary}\label{rep1}
\label{cor1oct} Let $H=(C_p\times C_p)\rtimes D_{p-1}$, and regard $C_p\times C_p$ as the translation group of the affine plane $\rm{AG}(2,p)$ over $\mathbb{F}_p$. Then there exists a coordinate system $(x,y)$ such that $C_p\times C_p$ is the group of all translations of $\rm{AG}(2,p)$ and $D_{p-1}=\langle U,V \rangle$ is the subgroup of $\GL(2,p)$ as in Lemma \ref{le1oct}.
\end{corollary}
{}From now on we use the notation introduced in in Lemma \ref{le1oct} and Corollary \ref{rep2}. With respect to the coordinate system $(x,y)$, let $\tau_{a,b}=(x,y)\to(x+a,y+b)$ for $a,b\in \mathbb{F}_p$, and
 $$
T_1=\{\tau_{a,a}|a\in \mathbb{F}_p\},\quad T_2=\{\tau_{a,-a}|a \in \mathbb{F}_p\}.
$$
A straightforward computation proves the following assertions.
\begin{lemma}\label{rep2}

\begin{itemize}
\item[(i)] $\tau_{1,1}T_2=T_2\tau_{1,1}$,\,\,$\tau_{1,1}V=V\tau_{1,1}$.
\item[(ii)] $W\tau_{a,b} W = \tau_{-a,-b}$.
\item[(iii)] $VT_2=T_2V$,\, $WV=VW$.
\item[(iv)] $\langle \tau_{1,1},V,W \rangle = D_p \times \langle V \rangle $
where $D_p$ is a dihedral group of order $2p$.
\end{itemize}
\end{lemma}

We also need a number of results of independent interest about curves whose automorphism groups contain a subgroup isomorphic to $C_p\times C_p$.
\begin{proposition}\label{pr1oct}
Assume that $\Aut_{\mathbb{K}}(\cX)$ contains a subgroup $H=S\rtimes D_{p-1}$ where $S=C_p \times C_p$.
 Let $M_1,\ldots,M_{p+1}$ be the cyclic subgroups of order $p$ of $S$. If $\cX$ has genus $\textsf{g}=(p-1)^2$ then the following assertions hold.
\begin{itemize}
\item[(i)] $\cX$ is an ordinary curve.
\item[(ii)] Up to labeling the indices, the extension $\mathbb{K}(\cX)|\mathbb{K}(\cX)^{M_i}$ is unramified for each $i \in \{3,\ldots,p+1\}$.
\item[(iii)] $S$ has only two short orbits $\Omega_1$ and $\Omega_2$, each of size $p$. The points of $\Omega_i$ share the same stabilizer $M_i$, for $i \in \{1,2\}$,  and $M_1 \neq M_2$. Moreover, $\cX/M_1$ and $\cX/M_2$ are rational.
\end{itemize}
\end{proposition}
\begin{proof}
Let $\gamma$ be the $p$-rank of $\cX$. If $\gamma=0$, \cite[Lemma 11.129]{HKT} implies that every element of $S$ fixes only one point of $\cX$. Take $\sigma \in S$ and let $P \in \cX$ the fixed point of $\sigma$. Choose $\tau \in S\backslash \langle \sigma \rangle$. From $\sigma\tau=\tau\sigma$, it follows that $\tau(P)=P$. Hence $P$ is the only fixed point of $\tau$. Therefore, $S$ has only one fixed point. Since $S$ is a normal subgroup of $H$, $P$ is fixed by $H$. Thus $H/S$ is a cyclic group (see e.g. \cite[Theorem 11.49]{HKT}), which is a contradiction by $H/S \cong D_{p-1}$. If $\gamma=1$, then $p^2|(\textsf{g}-1)$ by \cite[Theorem 1(ii)]{nakajima1987}, which is not possible as   $\textsf{g}=(p-1)^2$. Hence $\gamma \geq 2$.

 Set $\bar{\cX}=\cX/S$ and let $\bar{\gamma}$ be the $p$-rank of $\bar{\cX}$. The Deuring-Shafarevich formula (\ref{dsg}) applied to $S$ yields
\begin{equation}\label{equ1}
\gamma-1=p^2(\bar{\gamma}-1)+\sum_{i=1}^{k}(p^2-\ell_i),
\end{equation}
 where $k$ is the number of short orbits of $S$ and $\ell_1,\ldots,\ell_k$ are their sizes. If $k=0$, then $\gamma-1=p^2(\bar{\gamma}-1)$, and thus $\bar{\gamma}>1$ by $\gamma \geq 2$. But this yields $p^2 \leq \gamma-1 \leq \textsf{g}-1=p(p-2)$, a contradiction. Therefore $k \geq 1$. Now, since $|S|=p^2$, we have $\ell_i \in \{1,p\}$ for each $i$. Thus, bearing in mind that $\textsf{g}\geq \gamma$, (\ref{equ1}) gives
\begin{equation}\label{equ2}
 p(p-2)\geq (\bar{\gamma}+k-1)p^2-kp,
\end{equation}
so $\bar{\gamma}=0$ and $k \leq 2$. Suppose that $k=1$. Then (\ref{equ1}) implies $\gamma-1=-\ell$, with $\ell \in \{1,p\}$, which is impossible. Therefore $k=2$.

Let $\Omega_1$ and $\Omega_2$ be the short orbits of $S$ such that $\ell_i=|\Omega_i|$ for $i \in\{1,2\}$. Then (\ref{equ1}) reads
\begin{equation}\label{equ3}
\gamma-1=p^2-(\ell_1+\ell_2).
\end{equation}
In particular, $\ell_i=p$ for at least one $i$. If $\ell_i=1$ for some $i$, then $S$ fixes only one point $P \in \cX$ and, arguing as in the beginning of the proof, this raises a contradiction. Hence $\ell_1=\ell_2=p$, which proves (i). Moreover, all points of $\Omega_i$ have the same stabilizer $M_i$ of order $p$, for $1\le i\le 2$. Each of the remaining subgroups of order $p$ of $S$ has no fixed point, and thus (ii) is proved.

Finally, denote by $\textsf{g}_i$ the genus of the curve $\cX/M_i$, for $i=1,2$. By Riemann-Hurwitz Formula (\ref{rhg}) applied to the cover $\cX \rightarrow \cX /M_i$,
$$
2\textsf{g}-2 \geq p(2\textsf{g}_i-2)+2p(p-1).
$$
Hence $\textsf{g}_i=0$ for $i=1,2$. It remains to show that $M_1 \neq M_2$. In fact, if $M_1=M_2$, then $M_1$ would have $2p$ fixed points, which is impossible by Deuring-Shafarevich formula (\ref{dsg}) applied to $\cX \rightarrow \cX /M_1$.
\end{proof}

\begin{proposition}\label{pr3oct} For a curve $\cY$ defined over $\mathbb{F}_q$, where $q$ is a power of $p$,  assume that $\Aut_{\mathbb{F}_q}(\cY)$ has a subgroup $C_p\times \langle \mu \rangle$ with an involution $\mu$.
If $\cY$ has genus $\textsf{g}=p-1$ then $\cY$ is ordinary and hyperelliptic.
\end{proposition}
\begin{proof}
 Let $\bar{\textsf{g}}$ be the genus of the quotient curve $\bar{\cY}:=\cY/C_p.$ By the Riemann- Hurwitz formula (\ref{rhg})  applied to the cover $\cY \rightarrow \bar{\cY}$,
\begin{equation}\label{rh3}
2\textsf{g}-2=p(2\bar{\textsf{g}}-2)+\lambda(2+k)(p-1),
\end{equation}
where $\lambda$ the number of fixed points of $C_p$ and $k=\sum_{i\geq 2}(|S_P^{(i)}|-1)$ for $S=C_p$ and a fixed point $P$ of $C_p$. Since $k\geq 0$,
\begin{equation}\label{rh3.1}
4(p-1)=2p\bar{\textsf{g}}+\lambda(2+k)(p-1),
\end{equation}
which gives $1 \leq \lambda \leq 2$ and $\bar{\textsf{g}}=0$. In particular, $\bar{\cY}$ has $p$-rank $\bar{\gamma}=0$.

Now, let $\textsf{g}'$ be the genus of the quotient curve $\cY^{\prime}=\cY/\langle \mu \rangle$. Since $\mu$ is an involution, the Riemann-Hurwitz formula (\ref{rhg})  applied to the cover $\cY \rightarrow \cY^{\prime}$ yields
\begin{equation}\label{rh3.2}
2\textsf{g}-2=2(2\textsf{g}^{\prime}-2)+\delta,
\end{equation}
where $\delta$ is the number of fixed points of $\mu$. Hence
\begin{equation}\label{rh3.3}
\textsf{g}^{\prime}=\textstyle\frac{1}{4}(2p-\delta).
\end{equation}
This shows that $\delta$ is an even number and that $0<\delta \leq 2p$.

Assume first $\lambda=1$. Let $P \in \cY$ be the unique fixed point of $C_p$ in $\cY$. Since $\mu$ commutes with all element of $C_p$, we have $\mu(P)=P$. As $\delta$ is even, this yields the existence of a fixed point $Q\in \cY$ of $\mu$ other than $P$. For a nontrivial element $\sigma\in C_p$, let
$$
C_p(Q)=\{Q,\sigma(Q),\ldots,\sigma^{p-1}(Q)\}
$$
the orbit of $Q$ under the action of $C_p$. Since $\sigma^{i}\mu=\mu\sigma^{i}$ for all $i\geq 0$, we have $\mu(\sigma^{i}(Q))=\sigma^{i}(Q)$ for all $i \geq 0$. Thus $\mu$ fixes $C_p(Q)$ pointwise. Hence $\delta \geq p+1$. If $\mu$ fixed some more point $R \in \cY$ outside $C_P(Q)$, then $\mu$ would fix $C_p(R)$ pointwise, and then $\delta$ would be bigger than $2p$, contradicting (\ref{rh3.3}). Therefore $\delta=p+1$, and $\textsf{g}{^\prime}=(p-1)/4$. Finally, consider the quotient curve $\cY/G$ with $G=C_p \times \langle \mu \rangle$. Since $G$ is an abelian group, the curve $\cY/G$ is Galois covered by both $\bar{\cY}$ and $\cY^{\prime}$. This implies that $\cY/G$ is rational by L\"uroth's Theorem \cite[Proposition 3.5.9]{stichtenothbook}. From the Riemann-Hurwitz formula (\ref{rhg}) applied to the cover $\cY^{\prime} \rightarrow \cY/G$,
$$
2\textsf{g}^{\prime}-2=-2p+\alpha(2+j)(p-1),
$$
for some integers $\alpha,j \geq 0$. Therefore $5=2\alpha(2+j)$, a contradiction.

Therefore, $\lambda=2$. The Deuring-Shafarevich formula (\ref{dsg}) applied to the cover $\cY \rightarrow \bar{\cY}$ yields $\gamma=p-1$. Hence, $\cY$ is an ordinary curve. Let $P$ and $Q$ be the fixed points of $C_p$. Since $\mu$ commutes with every element of $C_p$, $\mu$ permutes the set of orbits of $C_p$. Therefore, one of the following possibilities arises:
\begin{itemize}
\item [(I)]  $\mu(P)=P$ and $\mu(Q)=Q$.
\item [(II)] $\mu(P)=Q$ and $\mu(Q)=P$.
\end{itemize}

Assume that (I) holds. The argument used to rule out case $\lambda=1$ also works here and shows $\delta=2$. Furthermore, $\textsf{g}^{\prime}=(p-1)/2$ by (\ref{rh3.3}). Now, since $\mu$ and $\sigma$ commute, $\Aut_{\mathbb{F}_q}(\cY^{\prime})$ has a subgroup $N_p$ of order $p$ acting on the points of $\cY^{\prime}$ as $C_p$ does on the set of orbits of $\mu$ in $\cY$. Let $P^{\prime},Q^{\prime} \in \cY^{\prime}$ be the points lying under $P$ and $Q$ respectively. Here, $P^{\prime} \neq Q^{\prime}$, otherwise (II) occurs. Thus $P^{\prime}$ and $Q^{\prime}$ are both fixed by $N_p$. Therefore, for the number $\beta$ of fixed points of $\cH^{\prime}$ by $N_p$,
\begin{equation}\label{beta}
\beta \geq 2.
\end{equation}
 Let $\textsf{g}^{\prime \prime}$ be the genus of the quotient curve $\cY^{\prime \prime}=\cY^{\prime}/N_p$. From the Riemann-Hurwitz formula (\ref{rhg}) applied to the cover $\cY^{\prime} \rightarrow \cY^{\prime \prime}$,
\begin{equation}\label{rh3.4}
2\textsf{g}^{\prime}-2=p(2\textsf{g}^{\prime \prime}-2)+\beta(2+l)(p-1),
\end{equation}
where $l \geq 0$. Hence $3(p-1)=2p\textsf{g}^{\prime \prime}+\beta(2+l)(p-1)$, which implies $\beta \leq 1$, contradicting (\ref{beta}). This rules out case (I).

Therefore, (II) holds. Let $R \in \cH$ be a fixed point of $\mu$. Here $R\neq P$ and $R\neq Q$. Again, the above argument involving the orbit of $R$ under the action of $C_p$ works and shows that $\delta \geq 2p$. Now, from (\ref{rh3.3}), we obtain $\textsf{g}^{\prime}=0$. This yields $\cY^{\prime} \cong \mathbb{P}^{1}(\mathbb{F}_q)$. It turns out that $\cY$ is a cover of degree $2$ of the projective line $\mathbb{P}^{1}(\mathbb{F}_q)$. Since $g>1$, $\cY$ is a hyperelliptic curve.
\end{proof}

\begin{proposition}\label{pr4oct} For a curve $\cX$ of genus $\textsf{g}=(p-1)^2$ defined over $\mathbb{F}_p$,  assume that
$\Aut_{\mathbb{F}_p}(\cX)$ has a subgroup $(C_p\times C_p)\rtimes D_{p-1}$.
Then $C_p \times C_p$ has a subgroup $T$ of order $p$ such that the quotient curve $\cX/T$ is $\mathbb{F}_p$-birationally equivalent to the curve $\cC$ with affine equation
$$y^p-y=ax+\frac{1}{x}, $$
where $a \in \mathbb{F}_p^{*}$.
\end{proposition}
\begin{proof}
Let $H=(C_p\times C_p)\rtimes D_{p-1}.$ By Corollary \ref{rep1}, $H\cong\langle \tau_{1,-1}, \tau_{1,1} \rangle \times \langle U,V \rangle.$ This isomorphism maps $\tau_{1,-1}, \tau_{1,1},U,V$ to $H$. Denote these images by  $\sigma_1, \sigma_2,$ $\alpha, \mu$, respectively. Let $T=\langle \sigma_{1} \rangle$, $T^{\prime}=\langle \sigma_2 \rangle$, and $\nu=\alpha^{\frac{p-1}{2}}$. By Lemma \ref{rep2}, each of the automorphisms $\sigma_2, \nu,\mu$ normalizes $T$. More precisely, Lemma \ref{rep2}(iv) implies $\tau_{1,-1}\not\in \langle \tau_{1,1},W,V\rangle$. Therefore, $T$ intersects $\langle \sigma_2, \nu, \mu \rangle$ trivially.
Thus the $\mathbb{F}_p$-automorphism group of the quotient curve $\cX/T$ has a subgroup isomorphic to $\langle \sigma_2, \nu, \mu \rangle$. Hence $\Aut_{\mathbb{F}_p}(\cX/T)$ has a subgroup $D_p \times \langle \iota \rangle$, where $\iota$ is an involution. From Dickson's classification of subgroups of $PGL(2,\mathbb{F}_p)$, see \cite[Theorem A.8]{HKT},
no subgroup of the $\mathbb{F}_p$-automorphism group of the rational function field $\mathbb{F}_p(x)$ has an abelian group of order $2p$. Therefore, the curve $\cX/T$ is not rational. Hence, from Proposition \ref{pr1oct}(ii), the cover $\cX \rightarrow \cX/T$ is unramified, and its genus is $p-1$ by the Riemann-Hurwitz formula (\ref{rhg}). Since $\cX$ has $p$-rank $\g=(p-1)^2$ by Proposition \ref{pr1oct}(i), the Deuring-Shafarevich formula applied to the cover $\cX \rightarrow \cX/T$ shows that the $p$-rank of $\cX/T$ is $p-1$. Moreover, as we have already observed,
$
\Aut_{\mathbb{F}_p}(\cX/T)$ has a subgroup $D_p \times \langle \iota \rangle$ containing $C_p \times \langle \iota \rangle.$
By Proposition \ref{pr3oct}, the curve $\cX/T$ is hyperelliptic.

We claim that the fixed points of $D_p$ in $\cX/T$ are $\mathbb{F}_p$-rational. Firstly, $D_p$ is the factor group $\langle \sigma_2,\nu \rangle/T$. According to Proposition \ref{pr1oct}, $C_p \times C_p$ has only two short orbits $\Omega_1$ and $\Omega_2$, each of size $p$. Since $\sigma_2$ commutes with $C_p \times C_p$, both $\Omega_1$ and $\Omega_2$ are preserved by $\sigma_2$. We show that the same holds for $\nu$. For a point $Q_1\in \Omega_1$, let $\langle \rho\rangle$ be the stabilizer of $Q_1$ in $C_p\times C_p$. By Proposition \ref{pr1oct}(iii), each point in $\Omega_1$ is fixed by $\rho$.  Moreover, $\langle \rho \rangle$ has as many fixed points as $p$ by the Deuring-Shafarevich formula applied to $\cX \rightarrow \cX/\langle \rho\rangle$. Thus $\Omega_1$ is the exact set of fixed points of $\rho$ (and $\langle \rho \rangle$).
From Lemma \ref{rep2}(ii), $\nu\rho=\rho^{-1}\nu$, and hence $\nu$ preserves $\Omega_1$. This remains true for $\Omega_2$.
The sets $\Omega_i$ are also orbits of $T$, for $i \in \{1,2\}$. Let $P_i \in \cX/T$ the points under $\Omega_i$, with $i \in \{1,2\}$. Then $P_1$ and $P_2$ are fixed points of $D_p$. Denote by $\lambda$ the number of fixed points of $D_p$ in $\cX/T$, and let $\bar{C_p}$ the unique subgroup of $D_p$ of order $p$. The number of fixed points of $\bar{C_p}$ is $\geq \lambda$. Since $(\cX/T)/\bar{C_p} \cong \cX/(C_p \times C_p)$ is rational by Proposition \ref{pr1oct}, the Deuring-Shafarevich formula applied to $\cX/T \rightarrow (\cX/T)/\bar{C_p}$ yields that $\bar{C_p}$ has only two fixed points. Thus $\lambda=2$, that is, $P_1$ and $P_2$ are the only fixed points of $D_p$.

Let $\Phi_p: \cX \rightarrow \cX$ be the $\mathbb{F}_p$-Frobenius map. For a point $P \in \Omega_1$, let $M_1$ be the stabilizer of $P$. As we have already proven, $\Omega_1$ is the exact set of fixed points of $M_1$. Since $M_1$ is defined over $\mathbb{F}_p$, each element of $M_1$ commutes with $\Phi_p$. Therefore, $\Phi_p$ preserves $\Omega_1$. This holds true for $\Omega_2$. Therefore $P_1$ and $P_2$ are $\mathbb{F}_p$-rational points.

Now, Proposition \ref{pr4oct} follows from Theorem \ref{tvgvv}.
\end{proof}

\begin{proposition}\label{lem2oct} Assume that $\Aut_{\mathbb{F}_p}(\cX)$ contains a subgroup $C_p \times C_p$ so that,
for some $\sigma\in C_p\times C_p$ and $a \in \mathbb{F}_p^{*}$, the quotient curve $\cX/\left\langle \sigma\right\rangle =\cC$ has affine equation
$$ y^p-y=ax+\frac{1}{x}. $$
Then the following hold.
\begin{itemize}
\item[(i)] $\mathbb{F}_p(\cX)^{C_p \times C_p}=\mathbb{F}_p(x)$.
\item[(ii)] If $\cX$ is an ordinary curve with genus $\textsf{g}=(p-1)^2$ then $C_p \times C_p$ contains a subgroup $M$ of order $p$ different from $\langle \sigma \rangle$ such that the quotient curve $\cF=\cX/M$ has affine equation  \begin{equation}\label{asr}
z^p-z=b+\frac{1}{x}
\end{equation}
with some $z \in \mathbb{F}_p(\cX)$ and $b \in \mathbb{F}_{p}$.
\end{itemize}
\end{proposition}
\begin{proof} (i) Since $$[\mathbb{F}_p(\cX):\mathbb{F}_p(x)]=p^2=[\mathbb{F}_p(\cX):\mathbb{F}_p(\cX)^{C_p \times C_p}],
$$ it is enough to prove that $\tau(x)=x$ for any $\tau \in (C_p \times C_p) \backslash \left\langle \sigma\right\rangle$. Since $\tau$ and $\sigma$ commute, $\tau$ induces a ${\mathbb{F}}_p$-automorphism $\tau'$ of $\mathbb{F}_p(x,y)$. If $\tau^{\prime}$ is trivial then $\tau(x)=x$, and (i) follows. Otherwise, $\tau^{\prime}$ has order $p$. From \cite[Lemma 2.1]{geervlugt1991}, $\Aut_{\mathbb{F}_p}(\cC)$ has only one cyclic subgroup of order $p$. Therefore, this group is $\langle \tau^{\prime} \rangle$. Since $\mathbb{F}_p(x,y)|\mathbb{F}_p(x)$ is a Galois extension of degree $p$, the fixed field of $\langle \tau^{\prime} \rangle$ is $\mathbb{F}_p(x)$. Hence $\tau(x)=x$, which finishes the proof of (i).

(ii) From (i), $\mathbb{F}_p(\cX)^{C_p \times C_p}=\mathbb{F}_p(x)$, that is,  $\cX / (C_p \times C_p)=\mathbb{P}^{1}(\mathbb{F}_p)$. From \cite[Proposition 3.2]{subrao1975}, $\cC$ is ordinary and has genus  $\textsf{g}^{\prime}=p-1$. The Deuring-Shafarevich Formula (\ref{dsg}) applied to $\langle \sigma \rangle$ shows that the extension $\mathbb{F}_p(\cX)|\mathbb{F}_p(\cC)$ is unramified. Let $P_0$ and $P_{\infty}$ be the zero and the pole of $x$. Then $P_0$ and $P_\infty$ are totally ramified in the extension $\mathbb{F}_p(\cC)|\mathbb{F}_p(x)$, and no other point of $\mathbb{P}^{1}(\mathbb{F}_p)$ ramifies; see \cite[Proposition 3.7.8]{stichtenothbook}. Therefore, both  $P_0$ and $P_\infty$ split in $p$ distinct points in $\cX$, and the remaining points of $\mathbb{P}^{1}(\mathbb{F}_p)$ split completely in $\cX$.

We show that $P_\infty$ is unramified in the extension $\mathbb{F}_p(\cX)^{M}|\mathbb{F}_p(x)$. For this purpose, let
$P^{1}_\infty \in \cX$ over $P_\infty$, and let $M$ be the stabilizer of $P^{1}_\infty$ in $C_p \times C_p$. Note that $|M|=p$, since $P_\infty$ splits in $p$ distinct points in $\cX$. Furthermore, each extension of $P_\infty$ in $\cX$ has the same stabilizer $M$. Therefore, $P_\infty$ splits completely in $\cX/M$, and the assertion follows.

Applying the Riemann-Hurwitz Formula (\ref{rhg}) to the extension $\mathbb{F}_p(\cX)|\mathbb{F}_p(\cX)^M$ yields
$$
2(p-1)^2-2 \geq p(2\bar{\textsf{g}}-2)+2p(p-1),
$$
where $\bar{\textsf{g}}$ is the genus of $\cX/M$. From this, $\bar{\textsf{g}}=0$. Here $[\mathbb{F}_p(\cX)^{M}:\mathbb{F}_p(x)]=p$, since
$$
p^2=[\mathbb{F}_p(\cX):\mathbb{F}_p(x)]=[\mathbb{F}_p(\cX):\mathbb{F}_p(\cX)^{M}][\mathbb{F}_p(\cX)^{M}:\mathbb{F}_p(x)].
$$
Now, from the Deuring-Shafarevich Formula (\ref{dsg}) applied to the extension $\mathbb{F}_p(\cX)^{M}|\mathbb{F}_p(x)$, $\mathbb{F}_p(x)$ has only one point that ramifies in $\mathbb{F}_p(\cX)^{M}|\mathbb{F}_p(x)$, and this  point must be $P_0$. Hence, Artin-Schreier theory, see \cite{hasse1935}, yields that the quotient curve $\cX/M$ has affine equation
$$
z^p-z=\frac{h(x)}{\prod_{i=1}^{k}(x-a_i)^{\lambda_i}}
$$
for some $z \in \mathbb{F}_p(\cX)$, $a_i \in \overline{\mathbb{F}}_p$, $\gcd(\lambda_i,p)=1$ for $i \in \{1,\ldots,k\}$, and $h(x) \in \mathbb{F}_p[x]$ where $\deg (h(x))-\sum_{i=1}^{k}\lambda_i$ is either negative, zero, or relatively prime to $p$. As $\bar{\textsf{g}}=0$ and $P_0$ is the unique point that ramifies in $\mathbb{F}_p(\cX)^{M}|\mathbb{F}_p(x)$, the only possibilities are $k=1$, $a_1=0$, $\lambda_1=1$ and $\deg (h(x)) \leq 1$; see \cite[Proposition 3.7.8]{stichtenothbook}. Thus, after an $\mathbb{F}_p$-scaling of $z$,  $\mathbb{F}_p(\cX)^{M}$ is the function field of the curve defined by (\ref{asr}).
\end{proof}


\section{Proof of Theorem \ref{main}}
We keep our notation introduced in Section \ref{gener}.

{}From Proposition \ref{pr4oct}, $C_p \times C_p$ contains a subgroup $T$ of order $p$ such that the quotient curve $\cX/T$ has function field $\mathbb{F}_p(x,y)$ with
$$y^p-y=ax+\frac{1}{x}, \ \ \ a \in \mathbb{F}_p^{*}.$$

From Proposition \ref{pr1oct}, the $p$-rank of $\cX$ is $\g=g=(p-1)^2$. Thus by Proposition \ref{lem2oct},  $\mathbb{F}_p(x,z)$, with $$z^p-z=b+\frac{1}{x}, \ \ \ b \in \mathbb{F}_p,$$ is a subfield of $\mathbb{F}_p(\cX)$. Therefore, the compositium $\mathbb{F}_p(x,y,z)$ of $\mathbb{F}_p(x,y)$ and $\mathbb{F}_p(x,z)$ is a subfield of $\mathbb{F}_p(\cX)$ such that
\begin{equation}\label{fiber}
\begin{cases}
y^p-y=ax+\frac{1}{x}\\
z^p-z=b+\frac{1}{x}.
\end{cases}
  \end{equation}
Therefore $\mathbb{F}_p(x,y,z)=\mathbb{F}_p(y,z)$ with
\begin{equation}\label{we}
(z^p-z-b)(y^p-y)-(z^p-z-b)^2=a.
\end{equation}	
 {}From Proposition \ref{lem2oct}, $\mathbb{F}_p(x,z)=\mathbb{F}_p(\cX)^{M}$ and $\mathbb{F}_p(x,y)=\mathbb{F}_p(\cX)^{T}$, where $M \neq T$ is a cyclic subgroup of $C_p \times C_p$ of order $p$. Thus
$$
\Gal\big(\mathbb{F}_p(\cX)|\mathbb{F}_p(y,z)\big)=\Gal\big(\mathbb{F}_p(\cX)|\mathbb{F}_p(\cX)^{M}\big) \cap \Gal\big(\mathbb{F}_p(\cX)|\mathbb{F}_p(\cX)^{T}\big)=M \cap T,
$$
which implies that $\Gal\big(\mathbb{F}_p(\cX)|\mathbb{F}_p(y,z)\big)=\{1\}$. Hence $\mathbb{F}_p(\cX)=\mathbb{F}_p(y,z)$.
\begin{lemma}
\label{lem2dic} Every subgroup $L$ of $\Aut_{\mathbb{F}_p}(\cX)$ of order $p$ is a subgroup of $C_p \times C_p$.
\end{lemma}
\begin{proof}
Let $L$ be a subgroup of order $p$. By Nakajima's bound \cite[Theorem 1(i)]{nakajima1987}, $|\langle (C_p \times C_p), L\rangle| \leq p^2$. Therefore $L$ is contained in $C_p \times C_p$.
\end{proof}
  For $\delta \in \mathbb{F}_{p^2}$ such that $\delta^p-\delta=2b$, the map $\varphi:(y,z) \mapsto (-y,-z+\delta)$ is an $\mathbb{F}_{p^2}$-automorphism of order $2$. Moreover, if
$\tau_{c,d}:(y,z) \mapsto (y+c,z+d)$ then $S=\{\tau_{c,d}\ \ | \ c,d \in \mathbb{F}_p\}$ is a subgroup of $\Aut_{\mathbb{F}_p}(\cX)$. From Lemma \ref{lem2dic}, $S=C_p \times C_p$. Furthermore, a direct computation shows that $\varphi \tau_{c,d} \varphi= \tau_{-c,-d}$ for all $c,d \in \mathbb{F}_p$. From the representation of $(C_p \times C_p) \rtimes D_{p-1}$ given in Corollary \ref{cor1oct},  $\Aut_{\mathbb{F}_p}(\cX)$ contains an involution $\varphi^{*}$ such that $\varphi^{*} \tau_{c,d} \varphi^{*}= \tau_{-c,-d}$ for all $c,d \in \mathbb{F}_p$. Hence $\psi=\varphi \varphi^{*}$ commutes with every element of $C_p \times C_p$.
\begin{lemma}
\label{lema2dic}$\psi \in C_p \times C_p$.
\end{lemma}
\begin{proof}
Assume on the contrary $\psi \notin C_p \times C_p$. Replace $\psi$ by a suitable power of $\psi$ so that the order of $\psi$ is a prime number $m$.

We show that the quotient curve $\bar{\cX}=\cX/\langle \psi \rangle$ has non-zero $p$-rank. Assume on the contrary that this $p$-rank $\bar{\g}$ is zero. Since $\psi\not\in S$ but $\psi$ commutes with all elements of $S=C_p \times C_p$, $\Aut_{\mathbb{F}_p}(\bar{\cX})$ has a subgroup $\bar{S}$ isomorphic to $S$ that acts on the points of $\bar{\cX}$ as $S$ does on the set of the orbits of $\langle \psi \rangle$ lying over those points. From $\bar{\g}=0$, the curve $\bar{\cX}$ has a point $\bar{P}$ fixed by $\bar{S}$; see \cite[Lemma 11.129]{HKT}. Let $\Lambda$ be the orbit of $\langle \psi \rangle$ lying over the point $\bar{P}$. If $\Lambda$ had size $1$, then $S$ would have an orbit of size $1$, which is not possible all short orbits of $S$ having size $p$. Thus $\Lambda$ has size $m$, and since $S$ permutes the points of $\Lambda$, this yields $p|m$, that is, $p=m$. But then $\Aut_{\mathbb{F}_p}(\cX)$ has a $p$-subgroup of order $p^3$, contradicting Nakajima's bound \cite[Theorem 1(i)]{nakajima1987}. Therefore $\bar{\g} \neq 0$. By \cite[Theorem 1(ii)]{nakajima1987}, we also have that $\bar{\g} \neq 1$. Hence $\bar{\g}\geq 2$. Note that since $\textsf{g}>1$, the Riemann-Hurwitz formula (\ref{rhg}) yields $\textsf{g} > \bar{\textsf{g}}$, where $\bar{\textsf{g}}$ is the genus of $\bar{\cX}$. Hence $(p-1)^2=\g=\textsf{g} > \bar{\textsf{g}} \geq \bar{\g}$, and therefore
\begin{equation}\label{eqg}
p^2>\frac{p}{p-2}(\bar{\g}-1).
\end{equation}
On the other hand, $\bar{\g} \geq 2$, and Nakajima's bound \cite[Theorem 1(i)]{nakajima1987} applied to $\bar{\cX}$ gives
$$
p^2 \leq \frac{p}{p-2}(\bar{\g}-1),
$$
contradicting (\ref{eqg}).
\end{proof}
{}From Lemma \ref{lema2dic}, $\varphi \varphi^{*}=\psi \in C_p \times C_p$, which shows that $\varphi \in \Aut_{\mathbb{F}_p}(\cX)$. Therefore $\delta \in \mathbb{F}_p$ whence $b=0$. By (\ref{we}), $\mathbb{F}_p(\cX)=\mathbb{F}_p(y,z)$, where
$$
(z^p-z)(y^p-y)-(z^p-z)^2=a.
$$
This equation is the same as (\ref{31octa}) up to the formal replacement $z/a$ with $x$ and $y-z$ with $y$. Therefore, $\cX$ is the Artin-Mumford curve. This completes the proof of Theorem \ref{main}.


\section*{Acknowledgments}
The first author was supported by FAPESP-Brazil, grant 2014/03497-6. The second author was partially supported by the Italian Ministry MURST, Strutture geometriche, combinatoria e loro applicazioni, PRIN 2013-16.

	\vspace{0,5cm}\noindent {\em Authors' addresses}:

\vspace{0.2 cm} \noindent Nazar ARAKELIAN \\
Instituto de Matem\'atica, Estat\'istica e Computa\c c\~ao Cient\'ifica
\\ Universidade Estadual de Campinas \\ Rua S\'ergio Buarque de Holanda, 651 \\
CEP 13083-859, Campinas SP
(Brazil).\\
 E--mail: {\tt nazar@ime.unicamp.br}

\text{}

\vspace{0.2cm}\noindent G\'abor KORCHM\'AROS\\ Dipartimento di
Matematica\\ Universit\`a Degli Studi Della Basilicata\\ Contrada Macchia
Romana\\ 85100 Potenza (Italy).\\E--mail: {\tt
gabor.korchmaros@unibas.it }

    \end{document}